\documentclass{siamltex}

\usepackage{amsfonts, amsmath, amssymb, psfrag, graphicx, bbm, mathrsfs, mathabx}

\newtheorem{remark}[theorem]{Remark}

\def\fin{\ifmmode{\Large$\diamond$}\else{\unskip\nobreak\hfil
    \penalty50\hskip1em\null\nobreak\hfil{\Large$\diamond$}
    \parfillskip=0pt\finalhyphendemerits=0\endgraf}\fi}

\def\be#1#2\ee{\begin{equation}\label{eq:#1}#2\end{equation}}
\def\req#1{{\rm(\ref{eq:#1})}}
\def\bdm  {\begin{displaymath}}
  \def\edm  {\end{displaymath}}
\def\bdmal{\begin{displaymath}\begin{aligned}}
    \def\edmal{\end{aligned}\end{displaymath}}

\mathchardef\Phi="0108
\mathchardef\Psi="0109
\mathchardef\Deltar="0101
\mathchardef\Omega="010A
\newcommand{\Schwartz}{{\mathscr{S}\!}}

\newcommand{\A}{\boldsymbol{A}}
\newcommand{\D}{\boldsymbol{D}}

\newcommand{\I}{\boldsymbol{I}}

\renewcommand{\O}{\boldsymbol{O}}

\newcommand{\F}{\boldsymbol{F}}
\renewcommand{\H}{\boldsymbol{H}}
\newcommand{\U}{\boldsymbol{U}}
\newcommand{\W}{\boldsymbol{W}}

\newcommand{\T}{\boldsymbol{T}}

\newcommand{\phihat}{{\widehat \varphi}}
\newcommand{\utilde}{{\widetilde u}}

\newcommand{\chat}{{\widehat c}}
\newcommand{\dhat}{{\widehat d}}
\newcommand{\what}{{\widehat w}}
\newcommand{\fhat}{{\widehat f}}

\newcommand{\hhat}{{\widehat h}}

\newcommand{\uhat}{{\widehat u}}

\newcommand{\V}{{{\mathscr V}}}
\renewcommand{\L}{{\mathscr L}}

\newcommand{\R}{{\mathord{\mathbb R}}}

\newcommand{\Wrho}{{{\mathcal W}_\varrho}}
\newcommand{\Ws}{W_\Sigma}
\newcommand{\Wsk}{W_{\Sigma,k}}

\newcommand{\uu}{{\boldsymbol{u}}}
\newcommand{\vv}{{\boldsymbol{v}}}
\newcommand{\ww}{{\boldsymbol{w}}}
\renewcommand{\gg}{{\boldsymbol{g}}}

\renewcommand{\ll}{\mbox{\mathversion{bold}$\ell$\mathversion{normal}}}

\newcommand{\Interval}{{\cal I}}

\newcommand{\Lrho}{{L_\varrho^\infty}}
\newcommand{\Lrhorrr}{{L_\varrho^\infty(\R^3)}}

\newcommand{\uHNC}{u_{\text{\tiny HNC}}}
\newcommand{\uLJ}{u_{\text{\tiny LJ}}}
\newcommand{\FLDL}{G_{\text{\tiny LDL}}}
\newcommand{\FLDLp}{G'_{\text{\tiny LDL}}}

\newcommand{\norm}[1]{\|#1\|}

\newcommand{\kB}{k_{\rm B}}
\newcommand{\Angstroem}{{\rm \AA}}
\newcommand{\kPa}{{\rm kPa}}

\newcommand{\gcm}{{\rm g}/{\rm cm}}

\newcommand{\rmd}{\,\mathrm{d}}
\newcommand{\dR}{\rmd\!R}

\newcommand{\dr}{\rmd r}

\newcommand{\domega}{\rmd\omega}
\newcommand{\rmi}{\mathrm{i}}
\newcommand{\eps}{\varepsilon}

\DeclareMathOperator*{\esssup}{ess\,sup}

\def\req#1{{\rm(\ref{eq:#1})}}

\makeatletter
\newcommand{\dupdots}{\mathinner{\mkern1mu\raise\p@
    \vbox{\kern7\p@\hbox{.}}\mkern2mu
    \raise4\p@\hbox{.}\mkern2mu\raise7\p@\hbox{.}\mkern1mu}}
\makeatother
\makeatletter
\newcount\c@MaxMatrixCols \c@MaxMatrixCols=10
\newenvironment{cmatrixc}{%
  \hskip -\arraycolsep\array{*\c@MaxMatrixCols c}%
}{%
  \endarray \hskip -\arraycolsep
}
\makeatother
\newenvironment{cmatrix}{\left[\cmatrixc}{\endmatrix\right]}

\title{A generalized Newton iteration for computing the solution of the
inverse Henderson problem}
\author{Fabrice Delbary, Martin Hanke, and 
        Dmitry Ivanizki\thanks{Institut f\"ur Mathematik, Johannes
    Gutenberg-Universit\"at Mainz, 55099 Mainz, Germany.
    The research leading to this work has been done within the 
    Collaborative Research Center TRR 146; corresponding funding 
    by the DFG is gratefully acknowledged.}}

\begin{document}
\maketitle
\sloppy

\begin{abstract}
We develop a generalized Newton scheme IHNC for the construction of 
effective pair potentials for systems of interacting point-like particles.
The construction is made in such a way that the distribution
of the  particles matches a given radial distribution function. The IHNC
iteration uses the hypernetted-chain integral equation for an approximate 
evaluation of the inverse of the Jacobian of the forward operator.

In contrast to the full Newton method realized in the Inverse Monte Carlo (IMC)
scheme, the IHNC algorithm requires only a single molecular dynamics 
computation of the radial distribution function per iteration step, 
and no further expensive cross-correlations. Numerical experiments are shown 
to demonstrate that the method is as efficient as the IMC scheme, and that it 
easily allows to incorporate thermodynamical constraints.
\end{abstract}


\begin{AMS}
  {\sc 65Z05, 82B21}
\end{AMS}

\section{Introduction}
\label{Sec:intro}
A common problem in material science is the quantification of
interactions between a set of given particles. For example,
in computer simulations of complex materials, where all sorts of
numerical multiscale techniques are inevitable tools to
treat relevant timescales and/or spatial resolutions
(cf., e.g., Potestio, Peter, and Kremer~\cite{PPK14}),
larger atomistic structures are often replaced by artificial particles,
so-called \emph{beads}, and the simulation of these beads requires the 
knowledge of effective interactions between them and other molecules or atoms.

In the simplest case one may assume that the beads are point-like 
particles, whose interactions are governed by a potential $u=u(r)$,
which only depends on the distance $r>0$ of each interacting pair of particles
and vanishes in the limit $r\to \infty$.
According to Henderson~\cite{Hend74} such a pair potential $u=u(r)$ 
is uniquely determined by the so-called \emph{radial distribution function} 
$g=g(r)$, which measures the number of particle pairs with a given distance 
in a homogeneous fluid in thermal equilibrium.  
The \emph{inverse Henderson problem} of computing the pair potential
from the given radial distribution is therefore exactly what needs to be solved
in order to settle the aforementioned problem in physical chemistry.

One of the difficulties with this problem is the fact that the associated map
\be{G}
   G\,:\, u\,\mapsto\, g\,,
\ee
which takes the pair potential onto the corresponding radial distribution 
function (for specified values of density and temperature of the fluid) 
is not given in
closed terms, but has to be evaluated numerically, using expensive molecular
dynamics simulations. It goes without saying that the inverse map $G^{-1}$ is 
not known, either. Methods for solving the inverse Henderson problem therefore
can be distinguished in two classes: one class uses closed form approximations
of $G$ or $G^{-1}$, respectively, most notably the hypernetted-chain or the 
Percus-Yevick approximations, cf., e.g., Ben-Naim~\cite{BenN06} or 
Hansen and McDonald~\cite{HaMcD13}; the other class uses iterative schemes 
which start from a certain educated guess $u_k$ of $u$, 
simulate the corresponding radial distribution function $g_k=G(u_k)$
and use this information to determine an improved approximation $u_{k+1}$ by 
some sophisticated update rule, proceeding in this manner until convergence. 
Most prominent representatives of 
the latter class are the Iterative Boltzmann Inversion (IBI) or 
Inverse Monte Carlo (IMC); cf., e.g., Mirzoev and Lyubartsev~\cite{MiLy13},
R\"uhle et al~\cite{RJLKA09}, or T{\'o}th~\cite{Toth07}.

In this paper we suggest a new method of the second class which, we believe,
combines the advantages of the two aforementioned schemes, 
namely the simplicity and robustness of IBI, and the rapid convergence of IMC 
for an appropriate initial guess. Our method is a generalized Newton 
iteration -- as opposed to IMC, which corresponds to the much more expensive
full Newton scheme for inverting \req{G} --
and we use the hypernetted-chain approximation to compute a simplified
derivative of $G$. We show by numerical examples for simulated
and measured radial distribution data that the method outperforms IBI and 
requires about the same number of iterations as does IMC, 
even when the density and the temperature of the fluid are
near a phase transition. We also demonstrate how to include 
thermodynamical constraints like a known value for the pressure of the system
into our scheme.
In this work we only treat the case of a homogeneous fluid of single particles;
we plan to show in a forthcoming paper how to extend the method to
binary mixtures.

The outline of this paper is as follows. In the following section we 
briefly summarize the necessary ingredients from statistical mechanics which
are fundamental for this work. Then, in Section~\ref{Sec:iH}, we derive the
approximation of the inverse of the Jacobian of $G$ which will be used
for our generalized Newton scheme. Section~\ref{Sec:well-posedness} presents
the mathematical core of this paper and is concerned with the well-posedness 
of different variants of our algorithm. Readers who are only interested in
the algorithms and in implementation details can skip this part without any
loss. In the subsequent two sections we then 
discuss numerical realizations and further extensions of these schemes; 
in particular, we show in Section~\ref{Sec:extensions} how to incorporate 
pressure constraints. Finally, numerical results for some benchmark systems 
are presented in Section~\ref{Sec:Numerics}. In an appendix we include a proof
for an extension of the classical Wiener lemma 
(cf., e.g., J\"orgens~\cite{Joer82}) to some weighted $L^\infty$ space,
which is needed for our mathematical analysis.

\section{Setting of the problem}
\label{Sec:setting}
Consider an ensemble of identical classical point-like particles 
in thermodynamical equilibrium, where the interaction of the particles 
is given in terms of a pair potential $u:\R^+\to\R$ 
of \emph{Lennard-Jones type}, i.e., there exist 
a \emph{core radius} $r_0>0$ and a parameter $\alpha>3$ such that 
\be{LJtype}
\begin{aligned}
    u(r)\phantom{|} &\,\geq\, a r^{-\alpha}\,, \quad &r&\leq r_0\,,\\[1ex]
   |u(r)| &\,\leq\, b r^{-\alpha}\,, \quad &r&\geq r_0\,, 
\end{aligned}
\ee
for suitable constants $a,b>0$. We assume that the number of particles and
the size of the spatial domain under consideration is so big that one can
treat this ensemble in the thermodynamical limit, i.e., as if it fills the
full space $\R^3$.
For our mathematical analysis 
we further assume that the counting density $\rho_0>0$ of the
ensemble is sufficiently small and the temperature $T>0$ is sufficiently large,
so that the system is in its so-called \emph{gas phase}, 
cf., e.g., Ruelle~\cite[p.~84]{Ruel69}.

The radial distribution function $g:\R^+\to\R^+$, referred to in the
introduction, measures the number of particle pairs at distance $r>0$, 
normalized in such a way that $g(r)\to 1$ as $r\to\infty$; 
see \cite{HaMcD13} for the precise definition of this function.
Then, as shown in \cite{Hank18a}, the map $G$ of \req{G},
which takes $u$ onto $g$, is well-defined and differentiable
in a certain neighborhood of $u$ with respect to the Banach space $\V$ of
perturbations $v$ of $u$, for which the corresponding norm 
\be{Vd}
   \norm{v}_\V \,=\, 
   \max\bigl\{ \norm{v/u}_{(0,r_0]}, \norm{\varrho v}_{[r_0,\infty)} \bigr\}
\ee
is sufficiently small\footnote{If $\Interval\subset\R$ denotes a real
interval then $\norm{\,\cdot\,}_\Interval$ refers to 
the supremum norm of real functions defined on this respective interval.};
here,
\be{walpha}
   \varrho(r) \,=\, (1+r^2)^{\alpha/2}\,, \qquad r\geq 0\,,
\ee
is a weight function associated with the parameter $\alpha$ of \req{LJtype}.

In \cite{Hank18c} it has been shown
that the so-called \emph{pair correlation function} $h=g-1$ for
a Lennard-Jones type pair potential given by \req{LJtype} belongs to the 
Banach space $\Lrho$ of functions $f\in L^\infty$ with finite norm
\be{Linftyrho}
   \norm{f}_{\Lrho} \,=\, \norm{\varrho f}_{(0,\infty)}\,,
\ee
where $\varrho$ is defined in \req{walpha}.
Since $\alpha>3$, the radially symmetric extension of any $f\in\Lrho$ 
to the full space $\R^3$ is absolutely integrable and has
a well-defined (three-dimensional) continuous Fourier transform.
This is important, because although $u$, $g$, and $h$ are defined as
functions of a positive argument $r>0$, 
they can be viewed as representations of radial
functions of a three-dimensional spatial variable in full space.
In particular, the Fourier transform of the corresponding extension of $h$
-- which is again radially symmetric and can therefore be represented by 
a function $\hhat:[0,\infty)\to\R$ by some slight abuse of the standard 
notation -- is used to define the \emph{structure factor}
\be{S}
   S(\omega) \,=\, 1 + \rho_0\hhat(\omega)\,, \qquad \omega\geq 0\,,
\ee
which is known to be continuous and nonnegative.


Going one step further, if $f_1,f_2\in\Lrho$, then the three-dimensional
convolution integral of their radially symmetric extensions to $\R^3$ 
is again a radial function
and -- as has also been shown in \cite{Hank18c} --
its representation (as a function defined in $\R^+$) again belongs to $\Lrho$; 
we adopt the notation $f_1*f_2$ for the resulting convolution product, 
which turns $\Lrho$ into a (commutative) Banach algebra.


\begin{proposition}
\label{Prop:hnc}
Let $u$ be a Lennard-Jones type pair potential \req{LJtype} with
parameter $\alpha>3$, and let the counting density $\rho_0$ of the ensemble 
be sufficiently small.
Using the pair correlation function $h$ of this ensemble and the above 
definition of the convolution product in $\Lrho$, define
\be{A}
   A:\Lrho\to\Lrho\,, \qquad A:f\mapsto \rho_0\,h*f\,.
\ee
Then $I+A$ is invertible in $\L(\Lrho)$, if the structure factor~\req{S} is 
strictly positive.
\end{proposition}

%


The proof of this result follows from a weighted version of the Wiener lemma,
stated and proved in the appendix, cf.~Lemma~\ref{Lem:Wiener}.

Under the assumptions of Proposition~\ref{Prop:hnc}
it follows in particular that the so-called \emph{Ornstein-Zernike relation}
\be{c}
   c \,+\, \rho_0 h*c \,=\, h
\ee
has a unique solution $c\in\Lrho$, known as the
\emph{direct correlation function}, cf.~\cite{HaMcD13}.
Then, with $\kB$ the Boltzmann constant and
\bdm
   \beta \,=\, \frac{1}{\kB T}
\edm
the inverse temperature, the \emph{hypernetted-chain approximation}
mentioned in the introduction states that
\be{HNC}
   g \,\approx\, e^{-\beta u + h-c}\,.
\ee
Historically \req{HNC} has been used to approximate $g$ without
lengthy molecular dynamics simulations, but by solving a (nonlinear)
integral equation instead. On the other hand, \req{HNC} can be solved for $u$
to provide an explicit approximation $\uHNC$ of the true pair potential, namely
\be{Y}
   \uHNC \,=\, U(g)
   \,=\, -\frac{1}{\beta} \,\log g \,+\, \frac{1}{\beta}\,(h-c)\,,
\ee
which only depends on quantities that are readily available
from the given radial distribution function.
    

We formally differentiate $U$ of \req{Y}
to determine the impact of small perturbations $g'$ of $g$ on $\uHNC$, namely
\be{Y-prime-tmp}
   U'(g)g' \,=\, -\frac{1}{\beta}\, \frac{g'}{g}
                 \,+\, \frac{1}{\beta}\,(g'-c')\,, 
\ee
where $c'$ is the derivative of $c$ with respect to $g$ or $h$, respectively:
Using \req{c} and the fact that $(\Lrho,*)$ is a Banach algebra,
we conclude that
\be{cprime-tmp}
   c' \,+\, \rho_0h*c' \,+\, \rho_0g'*c \,=\, g'\,.
\ee
Convolving this equation with $\rho_0 h$, adding the result to 
\req{cprime-tmp} again, and using the associativity and commutativity 
of the convolution product, we obtain
\bdm
   c' \,+\, 2\rho_0 h*c' \,+\, \rho_0^2\, h*h*c' \,+\, 
   \rho_0(c \,+\, \rho_0h*c)*g' \,=\, g'\,+\, \rho_0h*g'\,,
\edm
and inserting \req{c}, this yields
\bdm
   c' \,+\, 2\rho_0 h*c' \,+\, \rho_0^2\, h*h*c' \,=\, g'\,.
\edm   
With the operator $A$ of Proposition~\ref{Prop:hnc} the latter can be 
rewritten as
\bdm
   (I+A)^2 c' \,=\, g'\,,
\edm
showing that $c'\in\Lrho$ is well-defined when the structure factor is 
positive. Inserting this identity into \req{Y-prime-tmp}, we eventually obtain
\be{Y-prime}
   U'(g)g' \,=\, -\frac{1}{\beta}\, \frac{g'}{g}
                 \,+\, \frac{1}{\beta}\,\varphi\,, 
\ee
where
\be{varphi}
   \varphi \,=\, (I+A)^{-2}(2I + A) A\,g'\,.
\ee

\section{Generalized Newton schemes for the inverse Henderson problem}
\label{Sec:iH}
We now present iterative algorithms for an approximate solution of the
inverse Henderson problem, i.e., for determining a pair potential $\utilde$, 
for which the associated radial distribution function $G(\utilde)$ is close 
to the given data $g$ for specified values of $\rho_0$ and $\beta$.

One of the most successful methods of this kind is the 
\emph{Iterative Boltzmann Inversion} (IBI) 
\be{IBI}
   u_{k+1} \,=\, u_k \,+\, \frac{1}{\beta}\,\log\frac{g_k}{g}\,, \qquad
   k = 0,1,2,\dots\,,
\ee
originally suggested by Schommers~\cite{Scho73}. This method is widely used,
because it appears to be fairly robust. Soper, who redeveloped 
this scheme in \cite{Sope96}, gave some heuristic arguments to support this
observation. However, a rigorous convergence analysis is still lacking;
see \cite{Hank18c} for some preliminary results in this direction.

A certain shortcoming of IBI is that it may require quite a few iterations
to determine a sufficiently accurate potential.
In \cite{LyLa95} Lyubartsev and Laaksonen therefore proposed the 
Newton method
\be{IMC}
   u_{k+1} \,=\, u_k \,+\, G'(u_k)^{-1}(g-g_k)\,, \qquad g_k = G(u_k)\,,
\ee
$k=0,1,2,\dots$, as an alternative. 
In this scheme, now called \emph{Inverse Monte Carlo} (IMC),
the numerical evaluation of the Fr\'echet derivative of $G$
can be implemented by using higher order statistics of the 
ensemble corresponding to some integrated 3- and 4-particle distribution 
functions. As it requires longer forward simulations to achieve sufficiently 
accurate statistics of these higher order distribution functions, each IMC 
iteration is much more expensive than one step of IBI. 
Another shortcoming of IMC is the need to start the iteration with a 
fairly accurate initial guess. It is therefore sometimes recommended to first 
run a number of IBI steps before switching to IMC, cf., e.g., 
Mirzoev and Lyubartsev~\cite{MiLy13} or Murtola et al~\cite{MFKV07}.  

Here we propose a generalized Newton scheme,
where $G'(u_k)^{-1}$ in \req{IMC} is replaced by some approximation.
Note, for example, that the low-density approximation
\bdm
   G(u) \,\approx\, \FLDL(u) \,=\, e^{-\beta u}\,,
\edm
which is correct of order $O(\rho_0)$ as $\rho_0\to 0$, suggests to replace
\bdm
   G'(u_k)^{-1}g' \,\approx\, {\FLDLp(u)}^{-1}g'
   \,=\, -\frac{1}{\beta}\,e^{\beta u}g'
   \,\approx\, -\frac{1}{\beta}\,\frac{g'}{g}\,,
\edm
cf.~\cite{Ivan15}, and when using this approximation 
in \req{IMC} we arrive at the iterative scheme
\be{Ivanizki}
   u_{k+1} \,=\, u_k \,+\, \frac{1}{\beta}\,\frac{g_k-g}{g}\,, \qquad
   k = 0,1,2,\dots\,,
\ee
which is reminiscent of the IBI scheme \req{IBI}, because
\bdm
   \log\frac{g_k}{g} \,=\, \log\Bigl(1+\frac{g_k-g}{g}\Bigr)
   \,\approx\, \frac{g_k-g}{g}
\edm
for $g_k$ close to $g$. A more sophisticated
approximation of $G$, e.g., the one that is based on the hypernetted-chain 
approximation~\req{HNC}, 
which is correct of order $O(\rho_0^2)$ as $\rho_0\to 0$, cf.~\cite{HaMcD13},
may thus result in a powerful compromise between IBI and IMC.

To be specific, we propose to employ $U$ of \req{Y} and approximate 
\be{F-prime}
   G'(u_k)^{-1}g' \,\approx\, U'(g)g'
   \,=\, -\frac{1}{\beta} \frac{g'}{g}
         \,+\, \frac{1}{\beta}\,\varphi\,, 
\ee
cf.~\req{Y-prime},
where $g$ is the measured radial distribution function and $\varphi$ is
given by \req{varphi} with $A$ of \req{A}.
%
Inserting \req{F-prime} into \req{IMC} we thus obtain the iteration 
\begin{subequations}
\label{eq:HNCN}
\be{HNCN-rec}
   u_{k+1} \,=\, u_k \,+\, \frac{1}{\beta}\,\frac{g_k-g}{g}
   \,+\, \frac{1}{\beta}\,\varphi_k\,, \qquad k=0,1,2,\dots\,,
\ee
with
\be{varphi-k}
   \varphi_k \,=\, (I+A)^{-2}(2I+A)A\,(g-g_k)\,.
\ee
\end{subequations}
We call \req{HNCN} the \emph{hypernetted-chain Newton iteration} (HNCN).
Take note that this approach does \emph{not} involve a computation of
the hypernetted-chain approximation $\uHNC$ of \req{Y} itself;
the hypernetted-chain approximation is only used formally to determine
an approximate Newton inverse. 
Accordingly, when the iteration \req{HNCN} converges, i.e.,
when $u_k\to u$ and $g_k\to g$ as $k\to\infty$, then the limit $u$ is the
true solution of the Henderson problem for the given data.

Note that HNCN coincides with \req{Ivanizki} up to an additive correction 
term. The similarity between \req{Ivanizki} and IBI therefore
suggests to consider also the alternative IBI-type scheme
\be{IHNC}
   u_{k+1} \,=\, u_k \,+\, \frac{1}{\beta}\,\log\frac{g_k}{g}
   \,+\, \frac{1}{\beta}\,\varphi_k\,, \qquad k=0,1,2,\dots\,,
\ee
with $\varphi_k$ of \req{varphi-k},
which we call the \emph{inverse hypernetted-chain iteration} (IHNC).

We finally mention that IHNC and HNCN differ from the so-called
LWR scheme developed by Levesque, Weis, and Reatto~\cite{LWR85} and
rediscovered recently by Heinen~\cite{Hein18}: in our notation the LWR scheme
proceeds by computing
\bdm
   u_{k+1} \,=\, u_k \,+\, \frac{1}{\beta}\,\log\frac{g_k}{g}
   \,+\, \frac{1}{\beta}\bigl(g-g_k - c+c_k\bigr)\,, \qquad k=0,1,2,\dots\,,
\edm 
where $c$ is the direct correlation function~\req{c}, and $c_k$ is defined
accordingly via
\bdm
   c_k \,+\, \rho_0 h_k*c_k \,=\, h_k
\edm
with $h_k=g_k-1$. It is straightforward to verify that the LWR scheme 
can also be rewritten as
\bdm
   u_{k+1} \,=\, u_k \,+\, U(g) \,-\, U(g_k)
\edm
with $U$ of \req{Y}, hence the LWR update of the potential can be seen 
as the secant approximation of $U'(g)(g-g_k)$ used by the HNCN scheme.
While this may appear on first sight to be a minor difference only 
between the two schemes,
the tangent approximation turns out to be crucial to allow for 
subsequent extensions of the HNCN scheme described in 
Section~\ref{Sec:extensions}.

\section{Well-posedness of the IHNC and HNCN schemes}
\label{Sec:well-posedness}
We are now going to analyze the two new iterative schemes~\req{HNCN} and 
\req{IHNC} similar to the analysis of IBI in \cite{Hank18c}.
For this we work in the topology of the Banach space $\V$ defined in \req{Vd}.

\begin{proposition}
\label{Prop:IHNC}
Let $u$ be a Lennard-Jones type pair potential and $\rho_0$ be sufficiently
small. Moreover, assume that the structure factor~\req{S} is a strictly
positive function.
Then the IHNC iteration~\req{IHNC} is well-posed in the following sense:
If $\norm{u_0-u}_\V$ is sufficiently small, then $u_1$ is again a 
Lennard-Jones type pair potential, and there holds
\bdm
   \norm{u_1-u}_\V \,\leq\, C\norm{u_0-u}_\V
\edm
for some $C>0$, depending on $u$, $\rho_0$, and the inverse temperature $\beta$.
\end{proposition}

\begin{proof}
In the analysis of IBI in \cite{Hank18c} it has been shown that
\be{ineq1}
   \bigl\|\log(g_0/g)\bigr\|_\V \,\leq\, C \norm{u_0-u}_\V
\ee
for some constant $C>0$, cf.~\cite[(6.3)]{Hank18c}. Furthermore,
since $\Lrho$ is continuously embedded in $\V$ because of \req{LJtype}, and
since $A$ and $(I+A)^{-1}$ belong to $\L(\Lrho)$ by virtue of 
Proposition~\ref{Prop:hnc}, it follows from \req{varphi-k} that
\bdm
   \norm{\varphi_0}_\V \,\leq\, C \norm{\varphi_0}_\Lrho
   \,\leq\, C \norm{g_0-g}_\Lrho \,\leq\, C \norm{u_0-u}_\V
\edm
for some (other) constants $C>0$ that may be different in each of the 
individual terms;
here, the last inequality is borrowed from \cite[Theorem~5.3]{Hank18c}. 
Together with \req{IHNC} and \req{ineq1} we thus obtain the assertion.
\end{proof}

Concerning HNCN we have a similar result which is stated next, but this one
requires $u_0$ to be close to $u$ in the stronger norm of $\Lrho$.

\begin{theorem}
\label{Thm:HNCN}
Under the assumptions of Proposition~\ref{Prop:IHNC}
the HNCN iteration~\req{HNCN} is conditionally well-posed 
in the following sense: If $\norm{u_0-u}_\Lrho$ is sufficiently small,
then $u_1$ is again a Lennard-Jones type pair potential, and there holds
\bdm
   \norm{u_1-u}_\Lrho \,\leq\, C\norm{u_0-u}_\Lrho
\edm
for some $C>0$, depending on $u$, $\rho_0$, and the inverse temperature $\beta$.
\end{theorem}

\begin{proof}
According to \req{HNCN} there holds
\bdm
   u_1 - u \,=\, u_0-u \,+\, \frac{1}{\beta}\,\frac{g_0-g}{g} \,+\,
                 \frac{1}{\beta}\,\varphi_0\,,
\edm
where 
\bdm
   \norm{\varphi_0}_\Lrho \,\leq\, C\norm{u_0-u}_\Lrho
\edm
for some constant $C>0$ by virtue of \req{ineq1}, because $\Lrho$ is
continuously embedded in $\V$. It therefore remains to prove that
\be{goalproofPY-N}
   \Bigl\|\frac{g_0-g}{g}\Bigr\|_\Lrho \,\leq\, C\norm{u_0-u}_\Lrho
\ee
for some (other) suitable $C>0$. 

Consider first a fixed radius $r\geq r_0$. We rewrite
\bdm
   g(r) \,=\, y(r) e^{-\beta u(r)} 
\edm
in terms of the \emph{cavity distribution function} $y$, 
compare \cite{HaMcD13},
which is known to be bounded away from zero for small enough density $\rho_0$
according to Proposition~3.1 in \cite{Hank18c}. It follows that $g$ is
bounded away from zero for $r\geq r_0$, and hence,
there exist positive constants $C>0$ such that
\be{goalproofPY-N-1}
\begin{aligned}
   \varrho(r)\, \Bigl|\frac{g_0(r)-g(r)}{g(r)}\Bigr|
   &\,\leq\, C \varrho(r) \,\bigl|g_0(r)-g(r)\bigr|
    \,\leq\, C \norm{g_0-g}_\Lrho\\[1ex]
   &\,\leq\, C \norm{u_0-u}_\Lrho\,, \qquad r\geq r_0\,;
\end{aligned}
\ee
compare~\req{ineq1} again for the final estimate.

For a fixed radius $r$ with $0<r\leq r_0$, on the other hand, we
use the cavity distribution functions $y_0$ and $y$ corresponding to $u_0$
and $u$, respectively, and rewrite
\bdm
   \frac{g_0(r)-g(r)}{g(r)}
   \,=\, \frac{e^{\beta u(r)}\bigl(g_0(r)-g(r)\bigr)}{y(r)}\,.
\edm
Since $y$ is bounded away from zero we deduce from the mean value theorem that
\bdmal
   \varrho(r)\, \Bigl|\frac{g_0(r)-g(r)}{g(r)}\Bigr|
   &\,\leq\, C \,e^{\beta u(r)}\,\bigl|g_0(r)-g(r)\bigr|\\[1ex]
   &\,\leq\, C \Bigl( \bigl|y_0(r)-y(r)\bigr|
                        \,+\, g_0(r)\,\bigl|e^{\beta u(r)} - e^{\beta u_0(r)}\bigr|
                 \Bigr) \\[1ex]
   &\,=\, C \Bigl( \bigl|y_0(r)-y(r)\bigr|
                     \,+\, \beta g_0(r)e^{\beta \utilde}\,\bigl|u(r)-u_0(r)\bigr|
                 \Bigr)
\edmal
for some $C>0$ independent of $r$ and some $\utilde$ between $u_0(r)$ and 
$u(r)$. Note that the latter implies that
\bdm
   \tilde u \,\leq\, u_0(r) \,+\, \bigl|u_0(r)-u(r)\bigr|
   \,\leq\, u_0(r) \,+\, \norm{u_0-u}_\Lrho\,.
\edm 
Since the cavity distribution function in $L^\infty(\R^+)$ 
depends locally Lipschitz continuously on the pair potential in $\Lrho$ 
(see Proposition~3.1 in \cite{Hank18c}) it follows that
\bdmal
   \varrho(r)\, \Bigl|\frac{g_0(r)-g(r)}{g(r)}\Bigr|
   &\,\leq\, C \,\norm{u_0-u}_\Lrho 
             \Bigl(1+\beta g_0(r)e^{\beta \utilde}\Bigr)\\[1ex]
   &\,\leq\, C \,\norm{u_0-u}_\Lrho 
             \Bigl(1+\beta y_0(r)e^{\beta \norm{u_0-u}_\Lrho}\Bigr)\\[2ex]
   &\,\leq\, C \,\norm{u_0-u}_\Lrho\,, \qquad 0<r\leq r_0\,,
\edmal
for some suitable constants $C>0$, provided that $\norm{u_0-u}_\Lrho$ is 
sufficiently small. This being independent of $r\in(0,r_0]$,
we have thus achieved to establish
\req{goalproofPY-N-1} also for $0<r\leq r_0$, and hence the proof 
of \req{goalproofPY-N} is done.
\end{proof}

Theorem~\ref{Thm:HNCN} indicates that the HNCN iteration requires a better 
initial approximation of the true potential within the core region 
$0 < r \leq r_0$ than IHNC. Nevertheless, as shown
in \cite{Hank18c}, if the data $g$ are exact, then the potential of mean force,
\be{PMF}
   u_0 \,=\, -\frac{1}{\beta} \log g\,,
\ee
which is often taken as initial guess in practice,
does satisfy $u_0-u\in\Lrho$, which means that the assumptions of
Theorem~\ref{Thm:HNCN} are not too far-fetched.

\section{Numerical discretization}
\label{Sec:discretization}
Compared with IBI the only additional difficulty in a numerical 
implementation of HNCN and IHNC consists in computing $\varphi_k$ 
of \req{varphi-k}. To simplify notation let us denote by
\be{T}
   T \,=\, (I+A)^{-2}(2I+A)A
\ee
the operator occuring in \req{varphi-k}. Recall that $A$ corresponds to a 
three-dimensional convolution integral with $\rho_0$ times
the radially symmetric extension
of the pair correlation function $h=g-1$ as convolution kernel, cf.~\req{A}.
The natural framework for discretizing $A$ and $T$ is therefore the 
Fourier space, using the representation
\begin{subequations}
\label{eq:ahat}
\be{ahat-fwd}
   \fhat(\omega) 
   \,=\, \frac{2}{\omega} \int_0^\infty r\,f(r)\sin(2\pi r\omega)\dr
\ee
for the three-dimensional Fourier transform of the radially symmetric
extension of $f\in\Lrho$,
where $\omega>0$ is the absolute value of the three-dimensional frequency.
Likewise, we can compute $f$ from $\fhat$ by using the formula
\be{ahat-inv}
   f(r) \,=\, \frac{2}{r} 
              \int_0^\infty \omega\,\fhat(\omega) \sin(2\pi r\omega)\domega\,.
\ee
\end{subequations}
To implement $\varphi=Tf$ for $f\in\Lrho$ we therefore need to determine 
$\fhat$ and the corresponding representation $\hhat$ for $h$, form
\be{varphi-n-prime-2}
   \phihat \,=\, 
   \frac{2+\rho_0\hhat\phantom{{}^2}}{\bigl(1+\rho_0\hhat\bigr)^2}\,
   \rho_0\hhat\fhat\,,
\ee
and transform back using \req{ahat-inv} to obtain $\varphi$.

In order to achieve reasonable accuracy of the low frequencies of the
Fourier transform of $h$, the simulation box
and the particle count need to be sufficiently large.
Generally this implies that the radial distribution function is being
sampled on a larger radial interval than is used for tabulating the 
pair potential. To be specific, we will assume that the 
radial distribution function $g$ is given on a grid
\be{grid}
   \Delta \,=\, \{r_j=j\Deltar r\,:\, j=1,\dots,m\}
\ee
with $m$ grid points and spacing $\Deltar r>0$, and that $h(r)$ is negligible
for $r>r_m$.
On the other hand, the potentials $u_k$ are being tabulated on the subgrid
\be{gridprime}
   \Delta' \,=\, \{r_i=i\Deltar r\,:\, i=1,\dots,n\} \,\subset\,\Delta
\ee
with $n\leq m$ grid points and the understanding that $u_k(r)=0$ for
$r\geq r_n$.

For a generic function $f\in\Lrho$ 
which is vanishing for $r>r_m$ and which has been sampled on $\Delta$
the integral~\req{ahat-fwd} can be discretized with the 
trapezoidal quadrature rule.
Introducing the odd extension
\bdm
   \psi(r) \,=\, \begin{cases}
                    rf(r)\,, & r\geq 0\,, \\[1ex]
                    rf(-r)\,, & r<0\,,
                 \end{cases}
\edm
of $r\mapsto rf(r)$ to the whole real line (and to the extended grid with
nonpositive grid points $r_j$ with $j\leq 0$), and taking into account that
$\psi(r_j)=0$ for $|j|>m$, 
the quadrature approximation of \req{ahat-fwd} can be written as
\be{Shannon}
   \fhat(\omega) \,\approx\, \frac{1}{\rmi\omega}\,
   \Bigl(\Deltar r \sum_{j=-m}^{m+1} \psi(r_j) e^{-2\pi\rmi \omega r_j}\Bigr).
\ee
This approximation is in good agreement with
the true values of the Fourier transform of $f$ as long as
\bdm
   0 \,\leq\, \omega \,\leq\, \omega_* \,:=\, \frac{1}{2\Deltar r}\,,
\edm
provided that $f$ is negligible for $r>r_m$ and $\fhat$ is negligible 
for $\omega>\omega_*$, cf, e.g., Henrici~\cite[\S~13.3]{Henr86}.
Note that if the term in brackets in \req{Shannon} is to be evaluated 
at the $2(m+1)$ frequencies
\bdm
   \omega_l \,=\, \frac{l}{m+1}\,\omega_*\,, \qquad l=-m\,\dots,m+1\,,
\edm
then this can be implemented efficiently with a one-dimensional 
fast Fourier transform ({\sc fft}) of length $2(m+1)$, simultaneously for all 
these frequencies.


Alternatively, a matrix representation $\T\in\R^{m\times m}$ of the operator
$T$ of \req{T} can be assembled as
\be{TT}
   \T \,=\, \F^{-1}\H\F\,,
\ee
where $\F$ corresponds to the Fourier matrix which takes $[f(r_j)]_{j=1}^m$
onto $[\fhat(\omega_l)]_{l=1}^m$ given by \req{Shannon}, 
and $\H\in\R^{m\times m}$ is a diagonal matrix with the entries
\bdm
   h_{ll} \,=\,
   \frac{2+\rho_0\hhat(\omega_l)\phantom{{}^2}}
        {\bigl(1+\rho_0\hhat(\omega_l)\bigr)^2}
   \,\rho_0\hhat(\omega_l)\,, \qquad l=1,\dots,m\,,
\edm
on its diagonal; compare \req{varphi-n-prime-2}.
Note that the multiplication of $\T$ with the vector $\gg-\gg_k$
of samples of $g-g_k$ results in an $m$-dimensional vector with the values
of $\varphi_k$ of \req{varphi-k} over $\Delta$. If $\Delta'$ is a true
subset of $\Delta$, then we simply cut off the redundant entries when updating
the pair potential $u_k$, as it is done in IBI. 

\begin{remark}
\label{Rem:VOTCA}
\rm
We mention that common software like 
{\sc votca}\footnote{{\tt http://www.votca.org}}~\cite{RJLKA09} 
for running IBI typically comes with additional tricks for
pre- and postprocessing the relevant quantities, which
are not explicit in the recursion \req{IBI}. The same applies to
the new schemes HNCN and IHNC; more precisely the following items have
been addressed in our implementation of \req{HNCN} and \req{IHNC}:
\renewcommand{\labelenumi}{(\roman{enumi})}
\begin{enumerate}
\item
The simulated radial distribution functions
will be numerically zero in the core region $0<r\leq r_0$, in which case 
IBI as well as the new iterative schemes \req{HNCN} and \req{IHNC}
fail to produce a well-defined potential update for these radii; instead,
the potential $u_{k+1}$ needs to be extrapolated into the 
core region\footnote{The core region $r_0$ is chosen in each step
as the smallest grid point of \req{grid}, such that $g$ and $g_k$ are nonzero 
for every $r_j>r_0$.}
by some ad hoc scheme.
In our implementation we fit and extrapolate the computed values of $u_{k+1}$ 
in the core region to a function of the form $a'r^{-\alpha'}$ with 
appropriate positive parameters $a'$ and $\alpha'$.
\item
After each iteration the new potential $u_{k+1}$ is shifted by an
additive constant to satisfy $u_{k+1}(r_n)=0$, so that the extension of $u_{k+1}$ 
by zero for arguments $r>r_n$ is continuous.
\item
We have used 
{\sc gromacs}\footnote{{\tt http://www.gromacs.org}}~\cite{HKSL08} 
for the numerical computation of $g_k=G(u_k)$, with interpolated 
input values of $u_k$ on a grid which is ten times finer than $\Delta'$.
\end{enumerate}
We finally emphasize that our implementation of HNCN and IHNC uses
no postprocessing (e.g., smoothing) of the computed radial distribution 
functions, nor of the approximate potentials.
\fin
\end{remark}

\section{Extensions of the method}
\label{Sec:extensions}
Due to the many simplifying modeling assumptions, and also due to inevitable
noise in the given data, the inverse Henderson problem may not have a 
solution, and even when, it may not be appropriate to determine a 
pair potential $u$ which satisfies $G(u)=g$ exactly.
Rather, one should think of the problem as of an optimization problem
\bdm
   \text{minimize} \quad \norm{g-G(u)}
\edm
in some suitable norm, where the goal is to find an approximate minimizer
only. In the context of our generalized Newton approach
the obvious way of treating this minimization problem numerically is via
a Gauss-Newton type scheme, where each iteration consists of solving
the linearized minimization problem
\be{Gauss-Newton}
   \text{minimize} \quad \norm{g-g_k-G'(u_k)v}
\ee
before updating $u_{k+1} = u_k+v$; 
compare, e.g., Lyubartsev et al~\cite{LMCL10} or Murtola et al~\cite{MFKV07}. 
In view of \req{F-prime} we again propose
to replace $G'(u_k)$ by $U'(g)^{-1}$. With the same discretization as in
Section~\ref{Sec:discretization} this leads to the minimization problem
\be{Gauss-Newton-discrete}
   \text{minimize} \quad \norm{\W(\gg-\gg_k-\U^{-1}\vv)}_2
\ee
over $\vv\in\R^m$,
where $\norm{\,\cdot\,}_2$ denotes the standard Euclidean norm in $\R^m$,
$\W\in\R^{m\times m}$ is an appropriate nonnegative diagonal weighting matrix,
and
\bdm
   \U \,=\, -\frac{1}{\beta}\,\D^{-1} \,+\, \frac{1}{\beta}\,\T
\edm
is the discretized approximation of $U'(g)$, cf.~\req{F-prime};
here, $\D$ is the diagonal matrix with the samples of the given radial 
distribution function on its diagonal and $\T$ is defined in \req{TT}.

In view of Remark~\ref{Rem:VOTCA} this is not quite correct, though. 
As the samples
of the radial distribution function in the core region are numerically
zero, matrix $\D$ will fail to be invertible; but since the potential is
extended by extrapolation into the core region, anyway, we neither need to
keep track of the corresponding samples of $g$ nor of the respective 
function values of $u_k$. So, by some abuse of notation, we assume
in the sequel that the grid $\Delta$ only consists of the grid 
points $r_j$ in the exterior of the core region; we still
denote the number of grid points in $\Delta$ by $m$.

As has been mentioned in the previous section, $\Delta$
will typically have more grid points than $\Delta'$, and similar to above
we assume below that $\Delta'$ consists of the first $n<m$ grid points
$r_j$ of $\Delta$ outside the core region. If $\Delta'\subsetneq\Delta$,
then we only admit vectors $\vv\in\R^m$ for updating the pair potential 
which have zero entries for grid points $r_j\in\Delta\setminus\Delta'$.
Moreover, for several reasons we prefer to restrict admissible vectors
$\vv$ for \req{Gauss-Newton-discrete} somewhat further by substituting
\bdm
   \vv \,=\, A_0\ww 
\edm
with $\ww\in\R^{n-1}$ and
\be{P}
   \A_0 \,=\, \begin{cmatrix}
                 \A \\[1ex] \O
              \end{cmatrix},
   \qquad \text{where} \qquad
   \A \,=\, \Deltar r
            \begin{cmatrix}
                1 & 1 & \cdots & 1\\
                0 & 1 & \cdots & 1\\[-1ex]
                \vdots & 0 & \ddots & \vdots\\[-1.5ex]
                \vdots & \vdots & \ddots & 1\\
                0 & 0 & \cdots & 0
            \end{cmatrix}
   \in\R^{n\times (n-1)}
\ee
stands for a discrete (negative) antiderivative operator 
and $\O$ is an $(m-n)\times n$ zero block; 
accordingly, $\vv$ corresponds to a piecewise linear function
$v$ over $\Delta$ which is vanishing on $\Delta\setminus\Delta'$
and whose piecewise constant derivative on the grid intervals
of $\Delta'$ is given by the entries of $-\ww$.

We thus
determine the vector $\uu_{k+1}$ with the values of $u_{k+1}$ over $\Delta'$ 
by considering the weighted linear least squares problem
\begin{subequations}
\label{eq:HNCGN}
\be{HNCGN-a}
   \text{minimize} \quad \bigl\|\W(\gg - \gg_k - \U^{-1}\A_0\ww_k)\bigr\|_2 \,,
\ee
to be solved for $\ww_k\in\R^{n-1}$, and then update
\be{HNCGN-b}
   \uu_{k+1} \,=\, \uu_k \,+\, \A\ww_k\,.
\ee
This we call the \emph{hypernetted-chain Gauss-Newton iteration} (HNCGN).

One advantage of minimizing~\req{HNCGN-a} over $\ww=\ww_k$ rather than $\vv$ 
as in \req{Gauss-Newton-discrete} is that this adds some correlations to 
neighboring function values of the pair potentials; another advantage is that
we automatically respect the normalization condition $u_{k+1}(r_n)=0$, 
and therefore we avoid the extra shifting step mentioned in 
Remark~\ref{Rem:VOTCA}\,(ii).

With HNCGN it is easy to impose additional constraints on
$u_{k+1}$. As a simple example we treat the case that a certain
value $p$ for the pressure of the system is being imposed,
because this particular constraint has often been addressed in the literature 
as a possibility for improving the thermodynamical properties of 
coarse-grained models resulting from IBI or IMC iterations, 
cf., e.g., \cite{FKSL12,JGK06,MFKV07,PPK14,RPM03,WJK09}.
In the thermodynamical limit the pressure of the system 
is given by the virial integral
\bdm
   p \,=\, \frac{\rho_0}{\beta}
            \,-\, \frac{2}{3}\pi\rho_0^2\int_0^\infty u'(r)g(r)r^3\dr\,,
\edm
provided that the pair potential is differentiable and that its derivative
decays sufficiently rapidly near infinity; compare~\cite{HaMcD13}.
One way to enforce (approximately) the same pressure for the ensemble 
corresponding to the pair potential $u_{k+1}$ -- assuming that the simulated
radial distribution function $g_{k+1}$ is sufficiently close to the true one --
is by constraining $u_{k+1}$ to satisfy
\bdm
   \frac{2}{3}\pi\rho_0^2
   \int_0^\infty \Bigl(u_k'(r)-u_{k+1}'(r)\Bigr)g(r)r^3\dr
   \,\approx\, p-p_k \,,
\edm
where $p_k$ is the pressure corresponding to $u_k$; the latter can either
be evaluated within the simulation run for evaluating $G(u_k)$
or by numerical quadrature of the corresponding virial integral.
Since the entries $w_{i,k}$ of $\ww_{k}$ approximate the values
of $u_k'-u_{k+1}'$ over the interval $(r_i,r_{i+1})$,
the left-hand side of the previous equation can be discretized as
\bdm
   \frac{2}{3}\pi\rho_0^2 
   \sum_{i=1}^{n-1}w_{i,k}\,\frac{g(r_i)+g(r_{i+1})}{2}\frac{r_{i+1}^4-r_i^4}{4}
   \,=:\, \ll^T\ww_k
\edm
for a corresponding vector $\ll\in\R^{n-1}$, 
and this leads to a discrete constraint of the form
\be{HNCGN-constraint}
   \ll^T\ww_k \,=\, p-p_k
\ee
\end{subequations}
for all $\ww_k\in\R^{n-1}$, over which \req{HNCGN-a} is to be minimized. 

The standard recommendation for dealing with the constrained
minimization problem \req{HNCGN-a}, 
\req{HNCGN-constraint} numerically is to solve \req{HNCGN-constraint} 
for one of the entries in $\ww_k$, $w_{i_0,k}$ say,
and to use the resulting expression to eliminate this variable from 
\req{HNCGN-a}; cf., e.g., Bj\"orck~\cite{Bjor96}. To achieve maximal
stability $i_0$ should be the very index for which the corresponding element 
$\ell_{i_0}$ of $\ll\in\R^{n-1}$ has maximal modulus. Once $w_{i_0,k}$ has
been eliminated, \req{HNCGN-a} becomes an unconstrained minimization 
problem over the remaining entries of $\ww_k$, the solution of which is given
by the corresponding normal equation system, cf.~\cite{Bjor96}.
The final algorithm is slightly more expensive than IHNC, but the extra cost
is negligible compared to the overall costs of an individual iteration of
either of the schemes.

It remains to discuss the choice of the weighting matrix $\W$ in 
\req{Gauss-Newton}. A natural candidate is $\W=\I$, the $m\times m$
identity matrix. Alternatively, since it is known that $g-g_k=h-h_k\in\Lrho$
for some exponent $\alpha>3$, one could also think of using
$\W$ to enforce that the discrete approximation of $g-g_k$
shows a similar qualitative behavior for larger radii. In this case the
diagonal entries $w_{jj}$ of $\W$ should increase with increasing index,
e.g., 
\be{w-power}
   w_{jj} \,=\, (1+r_j^2)^\gamma\,, \qquad 1\leq j \leq m\,,
\ee
for some exponent $\gamma>0$.
However, we found that the choice \req{w-power} for $\gamma>0$ lent too much
flexibility to the values of $u_k(r)$ for radii $r$ near the core region, 
so that the computed potentials became worse eventually. 
In our numerical results in Subsection~\ref{Subsec:p-HNCGN} 
we therefore use $\W=\I$ throughout.


\section{Numerical results}
\label{Sec:Numerics}
We now present some numerical results to illustrate the performance of the
new methods as compared to IBI and IMC. For this we concentrate on the results
of IHNC; in all our tests we did not see significant differences
between IHNC and HNCN, but the theoretical results of 
Section~\ref{Sec:well-posedness} indicate that IHNC may be slightly more 
robust. 

Our benchmark problems include simulated data for a truncated and 
shifted Lennard-Jones potential as well as measured data for liquid argon
taken from the literature. 
We mention that for the latter problem, in particular, our mathematical 
assumption that the system be in its gas phase, is violated. As it turns out
this does not affect the applicability of our algorithms.
 
In all our 
numerical examples we use {\sc gromacs}, version 2016.3, to implement the 
forward operator $G$: to be specific, we simulate a canonical ensemble with 
$N=2000$ particles and determine the 
corresponding radial distribution function from the final 3500 frames.
For IMC the same frames are also used to set up the sensitivity matrix 
corresponding to $G'$. Remark~\ref{Rem:VOTCA} applies to our implementations 
of IBI and IMC in the same way.

\subsection{Truncated and shifted Lennard-Jones fluids near phase transitions}
\label{Subsec:Numerics-1}
Let
\be{uArgon}
   \uLJ \,=\, 4\eps\bigl((\sigma/r)^{12}-(\sigma/r)^6\bigr)\,, \qquad r>0\,,
\ee
be the classical Lennard-Jones potential with parameters $\eps,\sigma>0$.
Taking $\eps=\sigma=1$, i.e., working in reduced (dimensionless) units
with Boltzmann constant $\kB=1$,
we consider the \emph{truncated and shifted Lennard-Jones potential} 
\be{u-LJ-cut}
   u(r) \,=\, \begin{cases}
                 \uLJ(r)-\uLJ(2.5)\,, & 0<r<2.5\,,\\
                 0\,, & \phantom{0<}\ r\geq 2.5\,,
              \end{cases}
   \qquad
\ee
i.e., the Lennard-Jones potential is shifted, so that it becomes zero at 
$r=2.5$, and then extended continuously by zero for $r\geq 2.5$.
The corresponding ensemble is studied at two different state points, namely
\begin{itemize}
\item[(a)] the \emph{critical point} with counting density $\rho_0=0.304$ and 
temperature $T=1.316$, cf.~Smit~\cite{Smit92},
\item[(b)] a state point in the liquid phase close to the \emph{triple point} 
with counting density $\rho_0=0.8$ and temperature $T=1$, 
cf.~Hansen and Verlet~\cite{HaVe69}.
\end{itemize}
In both cases the radial distribution function is sampled on an
equidistant grid with mesh width $\Deltar r=0.02$; for state point (a)
we have data for $m=463$ grid points covering a radial interval $r\in(0,9.26]$,
for state point (b) we have $m=335$ grid points within the interval 
$(0,6.7]$. The latter interval is smaller than the former one,
because the density of the system is larger, and hence the simulation box
is smaller.
\begin{figure}
  \centerline{
    \includegraphics[width=6cm]{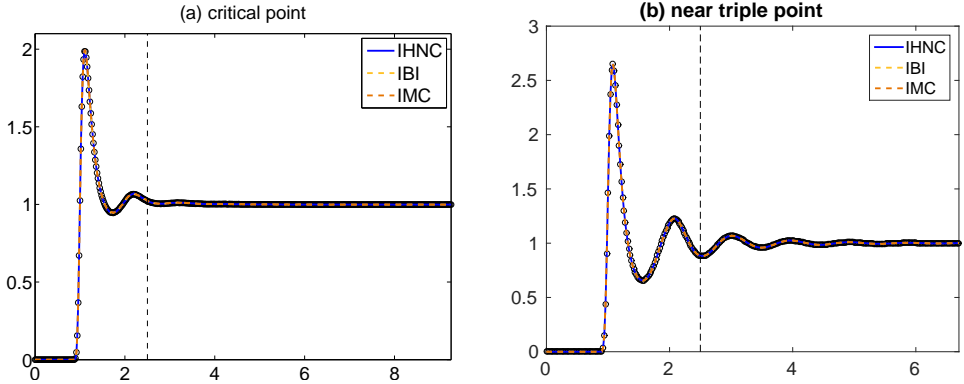}\qquad
    \includegraphics[width=6cm]{LJ_fast_tp_rdf.eps}}
  \caption{Truncated and shifted Lennard-Jones fluids: 
           radial distribution functions vs.\ radius}
  \label{Fig:LJ-rdfs}
\end{figure}
The given data are displayed as little circles in 
Figure~\ref{Fig:LJ-rdfs}. Note that the pair correlation function $h=g-1$ 
decays much faster at the critical point than near the triple point; as a 
consequence the inverse problem is much more difficult near the triple point.

To solve the inverse problem we tabulate the approximate potentials on 
the first $n=125$ grid points $r_i\in(0,2.5]$ of the same grid.
Because of the particular definition of the IBI and IMC schemes, 
cf.~\req{IBI} and \req{IMC}, 
only those $n$ grid points of the radial distribution function are used 
for these two methods; this radial interval is highlighted by the dashed lines
in Figure~\ref{Fig:LJ-rdfs}. For IHNC we use the full data displayed
in Figure~\ref{Fig:LJ-rdfs}. In all three iterative schemes 
the same potential~\req{PMF} of mean force is used as initial guess.

\begin{figure}
  \centerline{\phantom{l}
    \includegraphics[width=6cm]{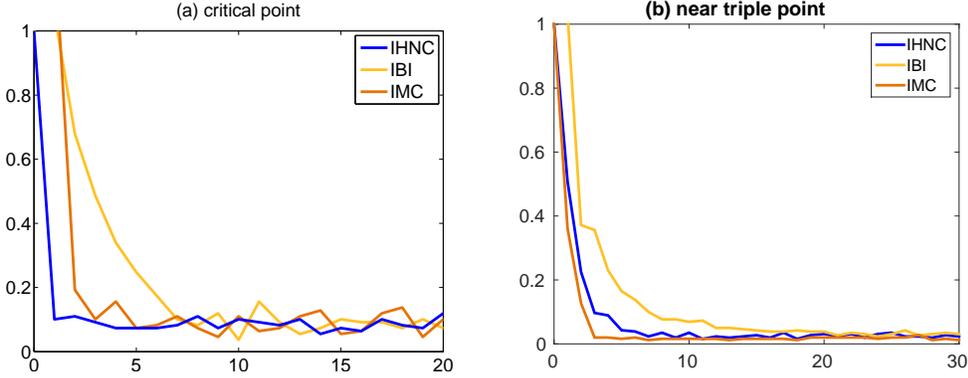} \qquad
    \includegraphics[width=6cm]{LJ_fast_tp_res_history.eps}}
  \caption{Truncated and shifted Lennard-Jones fluids: data fit vs.\ 
           iteration count}
  \label{Fig:LJ_res_history}
\end{figure}
The approximate radial distribution functions obtained by IBI, IHNC, and IMC,
respectively, are also shown in Figure~\ref{Fig:LJ-rdfs}. 
Essentially, all three functions are on top of each other in both plots, and 
they constitute perfect fits of the given data for each of the two state 
points. But IHNC and IMC require far less iterations to achieve this goal:
Figure~\ref{Fig:LJ_res_history} provides the corresponding iteration histories 
of the data fit, i.e., the graphs of the functions
\bdm
   k\,\mapsto\,\norm{G(u_k)-g}_\infty/\norm{G(u_0)-g}_\infty
\edm
for all three individual iteration schemes and for each of the two state 
points, respectively; here, $\norm{G(u_k)-g}_\infty$ measures the maximal
absolute error 
between \emph{all} given measurement data and the corresponding approximations. 
(For some obscure reason this measure of the data fit is slightly increasing
for IBI and IMC in the first iteration.)
From these plots it can be seen that IHNC requires about five
iterations at the critical point and eleven iterations near the triple point
to reach the global minimum of the data fit,
while IMC requires nine (or five?) iterations at the critical point and 
seven iterations near the triple point; the data fit of the two methods
is comparable, eventually.
IBI, on the other hand, needs ten iterations at the critical point
and more than twenty near the triple point.

\begin{figure}
  \centerline{
    \includegraphics[width=6cm]{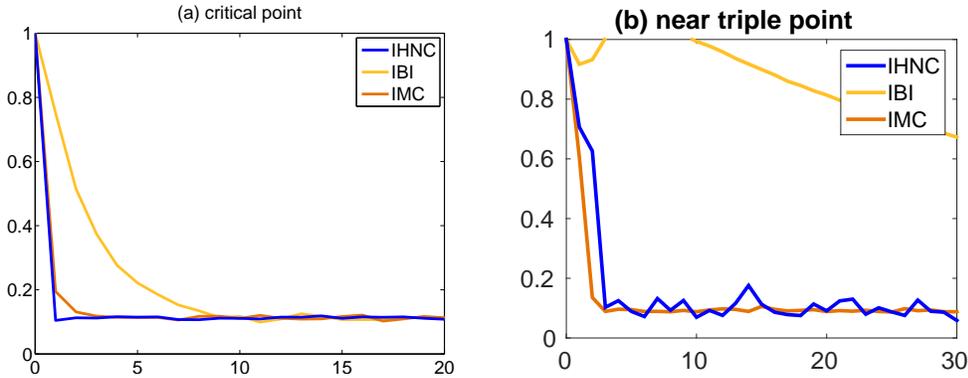} \qquad
    \includegraphics[width=6cm]{LJ_fast_tp_err_history.eps}}
  \caption{Truncated and shifted Lennard-Jones fluids: error \req{rel-error}
           vs.\ iteration count}
  \label{Fig:LJ_err_history}
\end{figure}
While the data fit is a straightforward indicator of the performance of 
the iterative schemes, the true error history is the really relevant quality 
measure. However, the latter is not available in practice.
It is the advantage of this particular example that the true solution is
known, so that the error history can be computed. 
For a particular potential $\utilde$ given on the grid $\Delta'$ 
we define the error measure
\begin{subequations}
\label{eq:epsilon}
\be{epsilon-a}
   \epsilon(\utilde) \,=\, 
   \Bigl(\Deltar r\sum_{i=1}^n g(r_i) \bigl(\utilde(r_i)-u(r_i)\bigr)^2 r_i^2
   \Bigr)^{1/2},
\ee
which approximates the weighted $L^2$ norm 
\be{epsilon-b}
   \epsilon(\utilde) \,\approx\,
   \Bigl(\int_0^\infty g(r) \bigl(\utilde(r)-u(r)\bigr)^2r^2\dr\Bigr)^{1/2}
\ee
\end{subequations}
of the error $\utilde-u$. This norm can be motivated by a more detailed 
analysis of the operator $G$, but this is beyond the scope of this paper; 
here we only emphasize that the factor $g$ in \req{epsilon-b} compensates 
for the divergence of the potentials as $r\to 0$.

\begin{figure}
  \centerline{
    \includegraphics[width=6cm]{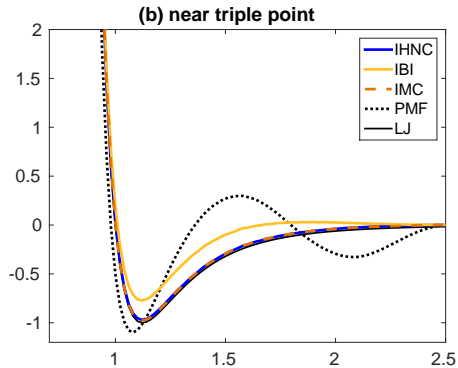}}
  \caption{Truncated and shifted Lennard-Jones fluid: reconstructed pair potentials vs.\ radius}
  \label{Fig:LJ_critical_potis}
\end{figure}
Figure~\ref{Fig:LJ_err_history} shows the relative error
\be{rel-error}
   k\,\mapsto\, \epsilon(u_k)/\epsilon(u_0)
\ee
as a function of the iteration count for all iterative schemes and both
state points, respectively.
This confirms that the particular iterates recommended above 
do indeed provide good approximations of the true 
truncated and shifted Lennard-Jones potential. Accordingly, IMC and IHNC
both converge very rapidly in much the same number of iterations, whereas 
IBI is doing significantly worse. To illustrate this further the corresponding 
reconstructions for the more difficult problem near the triple point are
shown in Figure~\ref{Fig:LJ_critical_potis}: This plot displays the 11th IHNC
iterate, the 7th IMC iterate and the 50th\,(!)\,IBI iterate, 
together with the true pair potential as a black solid line (marked ``LJ'') 
and the potential $u_0$ of mean force as a dotted line. 
As can be seen the IHNC and IMC approximations are hardly
distinguishable from the true truncated and shifted Lennard-Jones potential, 
while even after fifty iterations IBI is still relatively far off.

\subsection{Liquid argon}
\label{Subsec:Numerics-2}
As a second example we determine approximate pair potentials for argon,
using measurements by Schmidt and Tompson~\cite{ScTo68}
for a state point with temperature $T=-125^\circ$\,C and mass density
$0.982\,\gcm^3$ in the liquid phase near the 
critical point.\footnote{According to the US National Institute of Standards 
and Technology the critical point of argon is located at about temperature
$T=-122.3^\circ$\,C and mass density $0.536\,\gcm^3$; see
{\tt https://webbook.nist.gov/cgi/inchi?ID=C7440371\&Mask=4}.}
The corresponding data are given on an equidistant grid\footnote{For 
$r>10\,\Angstroem$ only every second data point is given in \cite{ScTo68};
the missing data have been filled in by linear interpolation.}
with $m=200$ grid points and 
mesh width $\Deltar r=0.1\,\Angstroem$. The approximate pair potentials are 
tabulated on the first $n=100$ grid points of this grid, and are taken to be
identically zero for $r>10\,\Angstroem$. 

\begin{figure}
  \centerline{\includegraphics[width=6cm]{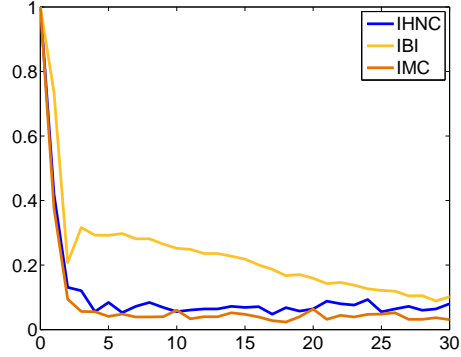}}
  \caption{Liquid argon: data fit vs.\ iteration count}
  \label{Fig:Ar_cp6_res_history}
\end{figure}

\begin{figure}
  \centerline{
     \includegraphics[height=4.8cm]{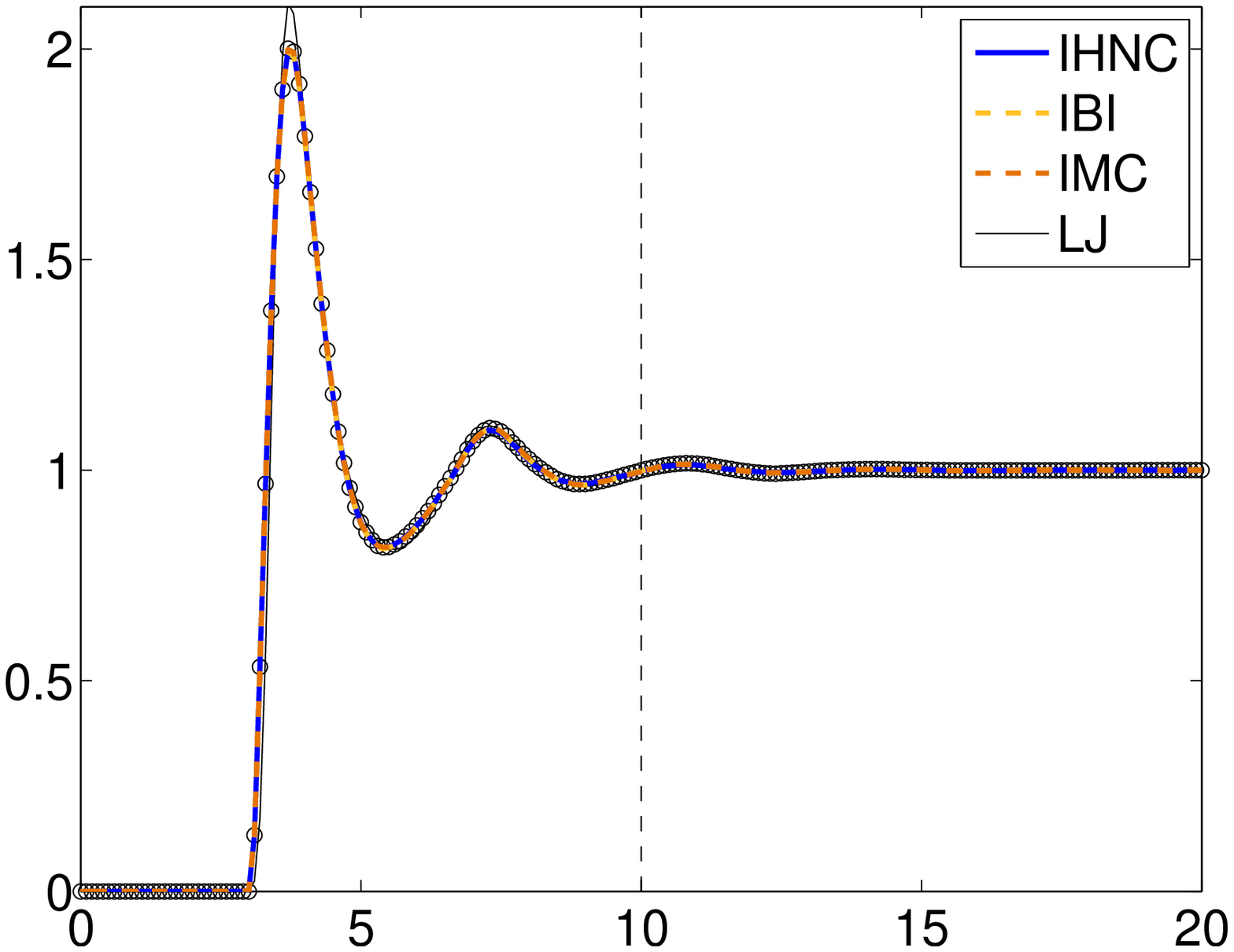} \qquad
     \includegraphics[height=4.8cm]{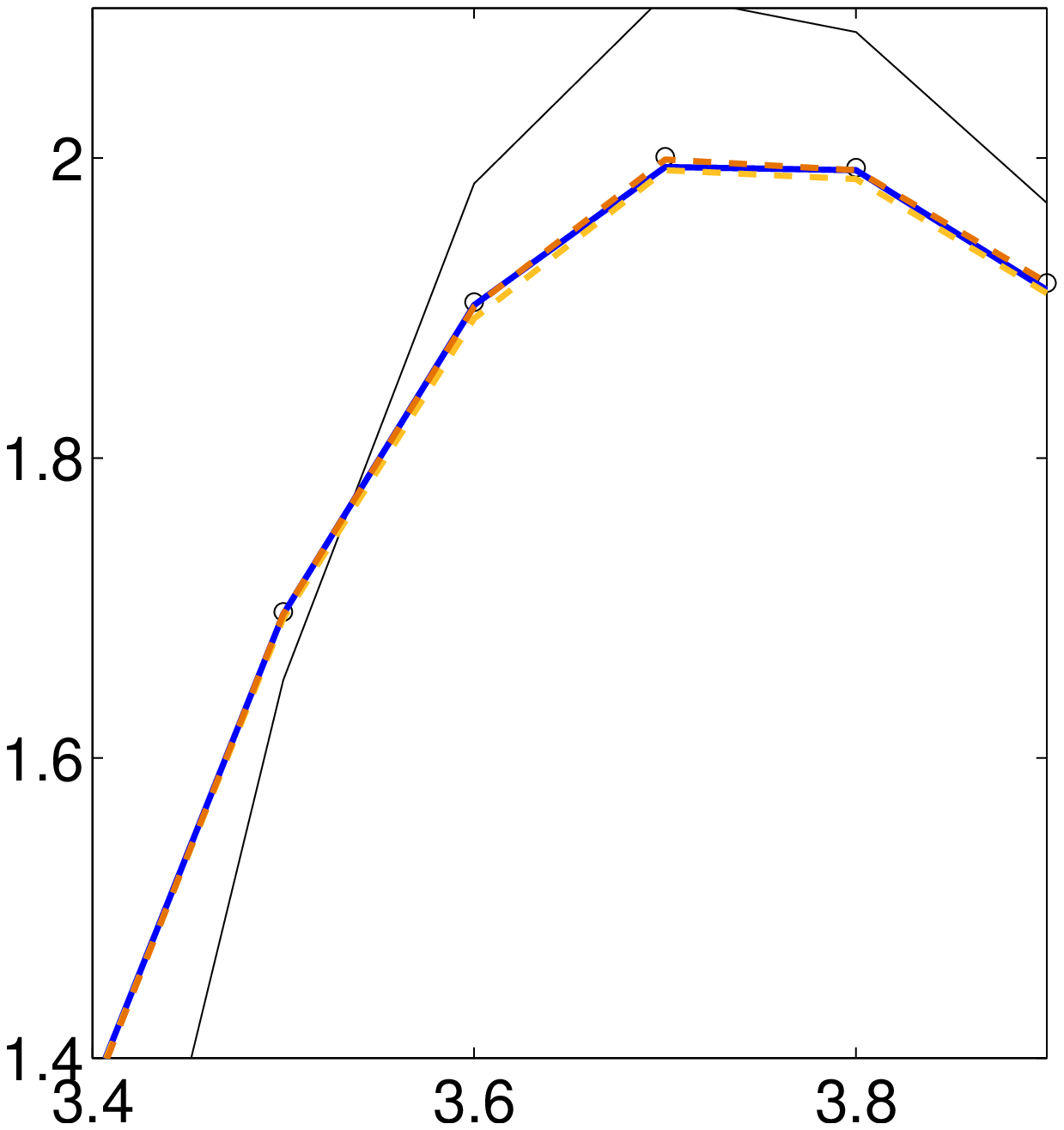}}  
  \caption{Liquid argon: radial distribution functions
           vs.\ radius (in {\rm$\Angstroem$}); detail view on the right}
  \label{Fig:Ar_cp6_rdf}
\end{figure}

\begin{figure}
  \centerline{\includegraphics[height=4.8cm]{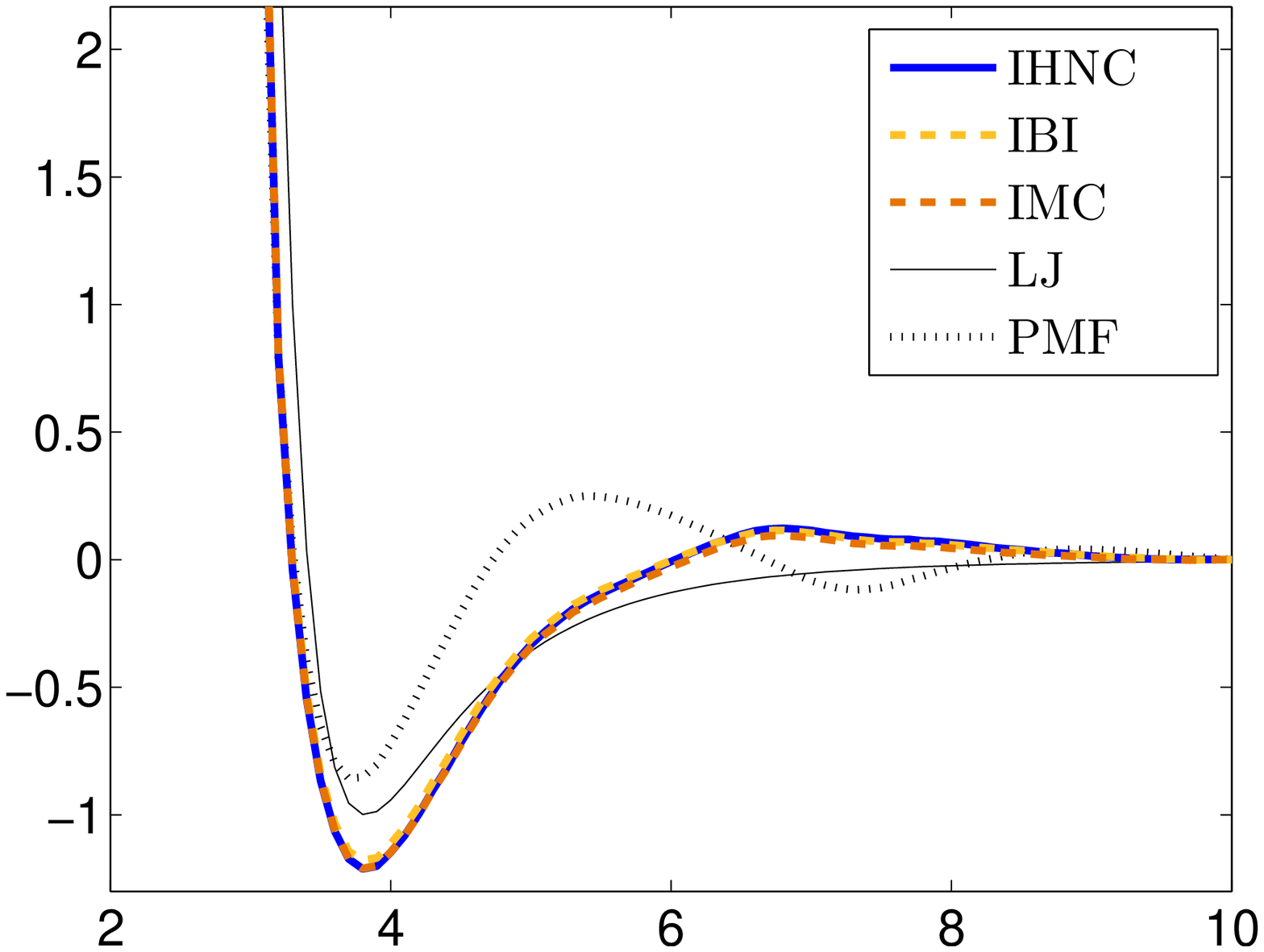}}  
  \caption{Liquid argon: approximate potentials
           (in units of $\eps$) vs.\ radius (in {\rm$\Angstroem$})}
  \label{Fig:Ar_cp6_potis}
\end{figure}

The iteration history shown in Figure~\ref{Fig:Ar_cp6_res_history} documents 
that, again, IHNC and IMC match the data much faster than IBI does: 
according to this plot six IHNC (five IMC) iterations should be sufficient, 
whereas IBI needs 39 iterations 
to achieve the same accuracy. 
Figure~\ref{Fig:Ar_cp6_rdf} presents the corresponding approximations of the
radial distribution function and Figure~\ref{Fig:Ar_cp6_potis} the 
corresponding potentials, together with the potential of mean force as dotted
line. As before, all computed approximations are very close to each other.


We have chosen argon as benchmark test case, because the interactions between 
argon atoms are widely considered to be well-described by the
Lennard-Jones pair potential~\req{uArgon} with parameters
\bdm
   \sigma=3.405\,\Angstroem \qquad \text{and} \qquad 
   \eps=119.8\,\kB\,{\rm J}\,,
\edm
cf., e.g., Tuckerman~\cite[p.~127]{Tuck10}.
This Lennard-Jones potential is given by the thin black line in 
Figure~\ref{Fig:Ar_cp6_potis}, but it differs quite a bit from our computed
pair potentials. In fact, the radial distribution function corresponding to 
this Lennard-Jones approximation, which has also been included in
Figure~\ref{Fig:Ar_cp6_rdf}, does not fit the measured data well, as can
easily be seen in the magnified detail in the right-hand plot.
Even the initial approximation from the potential
of mean force is doing better than that. So for this real-world example we
cannot trust this Lennard-Jones potential to be the ``ground truth'' 
to compare our numerical results to.

\subsection{The pressure constrained HNCGN scheme}
\label{Subsec:p-HNCGN}
Finally, we show some numerical results for $p$-HNCGN, i.e.,
the pressure constrained hypernetted-chain Gauss-Newton iteration
described in Section~\ref{Sec:extensions}. For this we use the same data set
for liquid argon as in the previous example and impose the corresponding value
$p=9918.7\,\kPa$
of the pressure reported by Mikolaj and Pings~\cite{MiPi67}. 

\begin{figure}
  \centerline{\includegraphics[width=6cm]{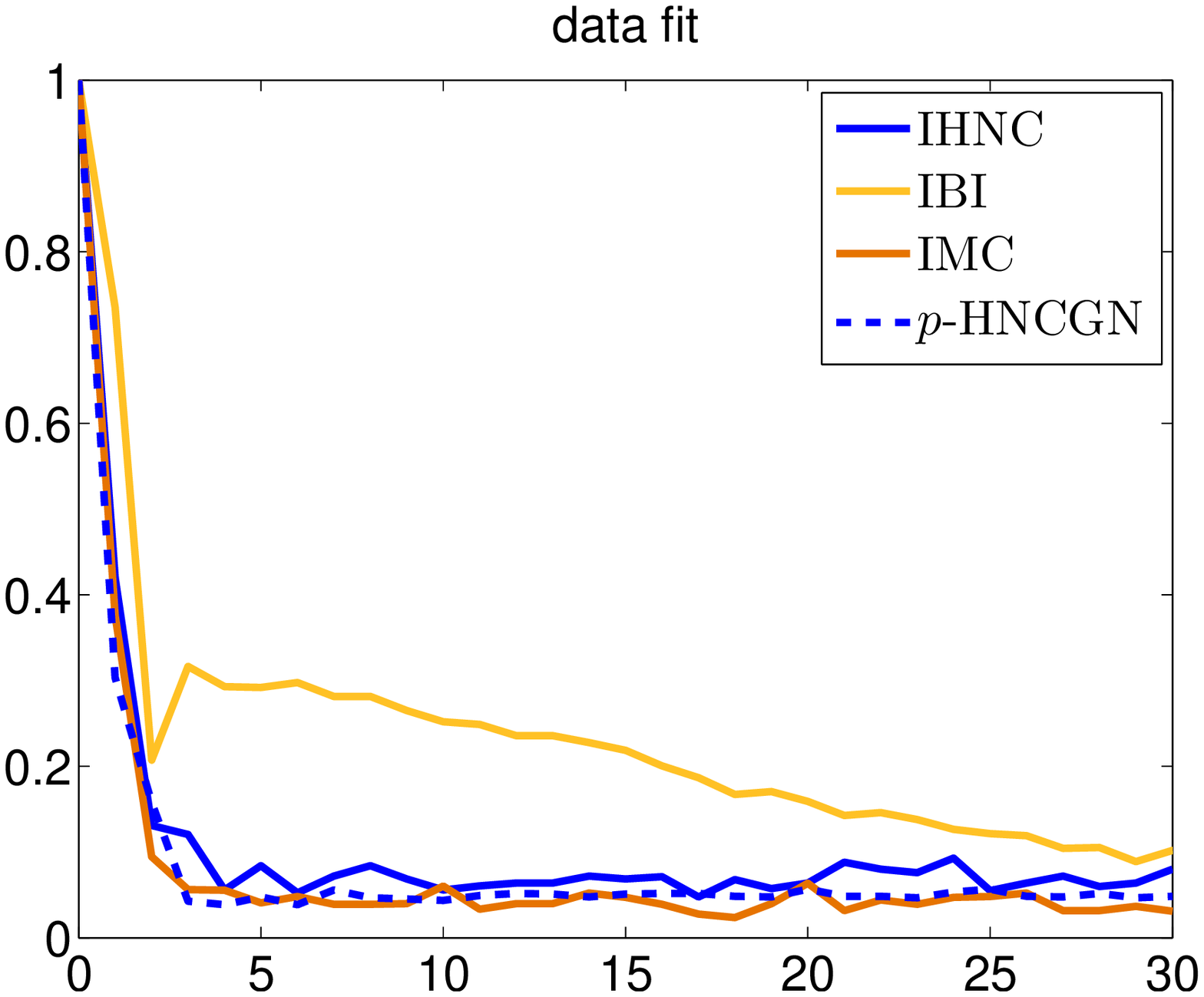}\qquad
              \includegraphics[width=6cm]{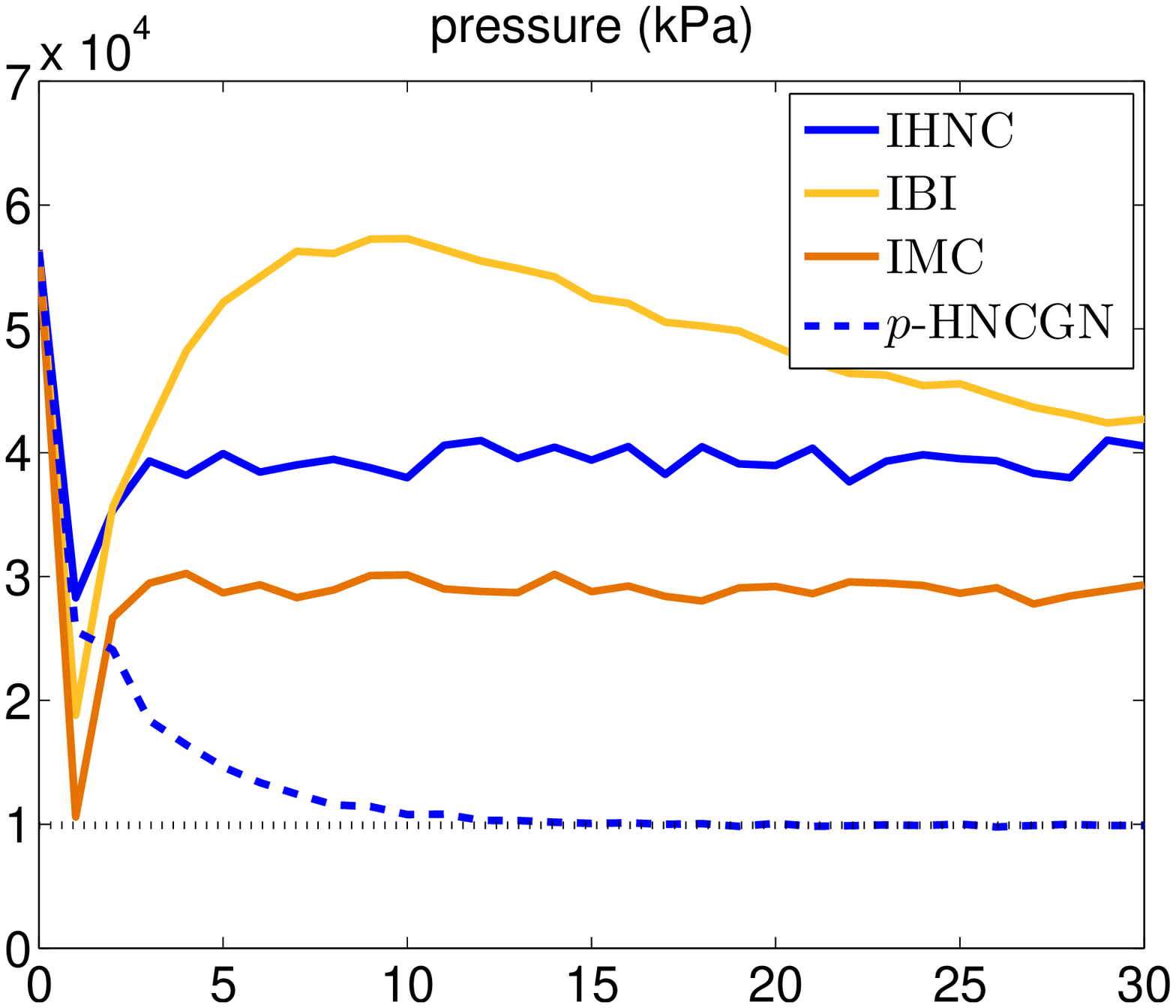}}
  \caption{Liquid argon: iteration history}
  \label{Fig:Ar_cp6p_it_history}
\end{figure}

\begin{figure}
  \centerline{\includegraphics[height=4.8cm]{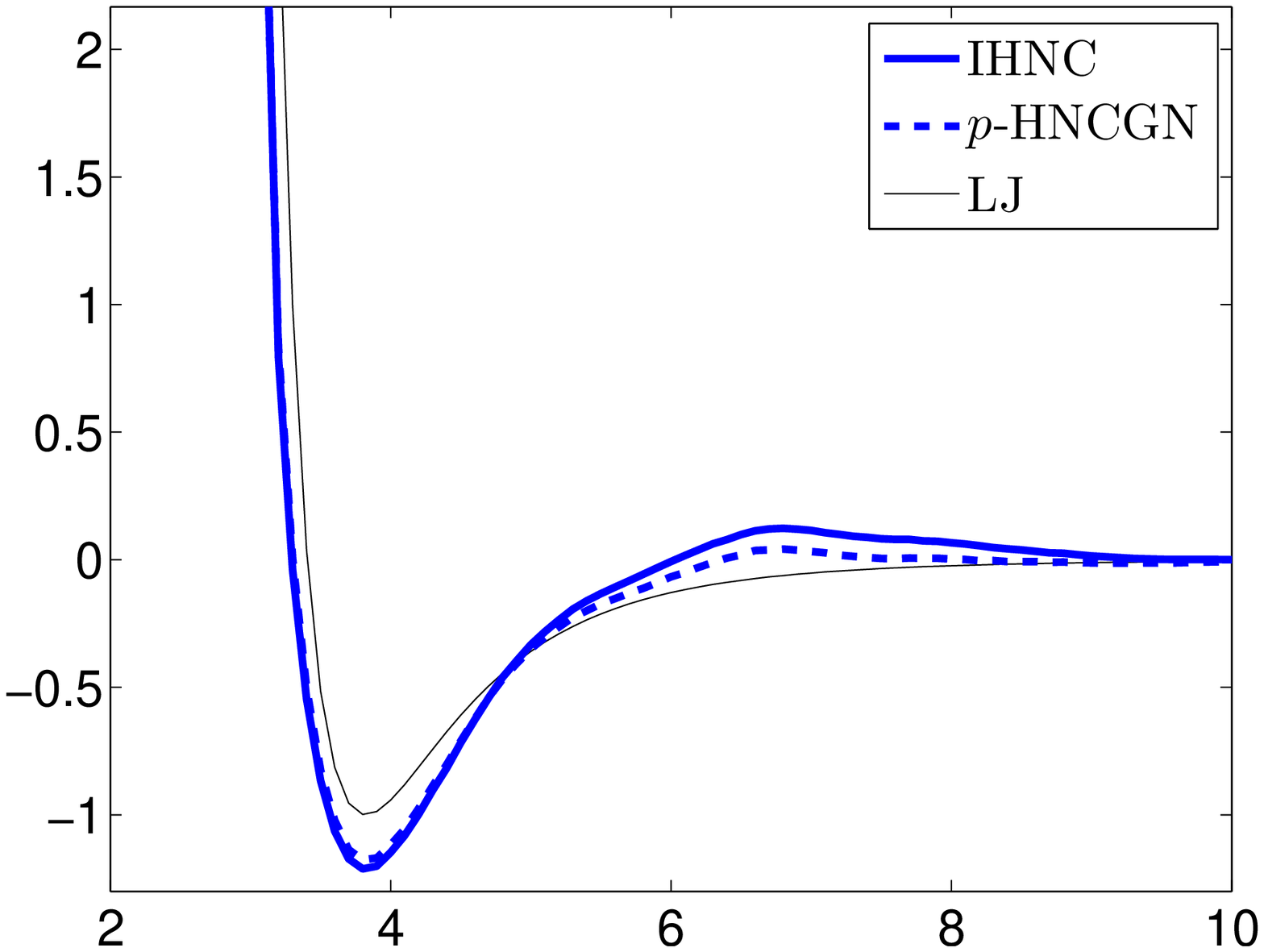}}  
  \caption{Liquid argon: approximate potentials
           (in units of $\eps$) vs.\ radius (in {\rm$\Angstroem$})}
  \label{Fig:Ar_cp6p_potis}
\end{figure}

Figures~\ref{Fig:Ar_cp6p_it_history} and \ref{Fig:Ar_cp6p_potis}
present the corresponding numerical results. 
In the left-hand plot of Figure~\ref{Fig:Ar_cp6p_it_history}
we recollect the data fit history of Figure~\ref{Fig:Ar_cp6_res_history} 
and add the corresponding graph for the performance of $p$-HNCGN: 
since the latter aims for a best possible fit of \emph{all} 200 data points
$g(r_j)$, it reaches a smaller value than all other competing methods. 

The right-hand plot of Figure~\ref{Fig:Ar_cp6p_it_history}, on the other hand,
displays the average pressure (as returned by {\sc gromacs}) of all 
corresponding ensembles for each individual iterate of the respective methods. 
The correct value of the pressure is indicated by the dotted horizontal line.
As can be seen, except for $p$-HNCGN all methods fail to reproduce this number
by a factor of three or more. $p$-HNCGN, on the other hand, 
achieves an excellent match of the target pressure after about 12 iterations. 

Assessing both plots of Figure~\ref{Fig:Ar_cp6p_it_history} we consider the
14th iterate of $p$-HNCGN to be ``optimal'', because it corresponds to the
first local minimum of the data fit after having reached a fairly accurate
value of the pressure. The corresponding pair potential is compared 
in Figure~\ref{Fig:Ar_cp6p_potis} with the Lennard-Jones reference and the 
IHNC potential from Figure~\ref{Fig:Ar_cp6_potis}. It can be seen that
the match of the pressure has a significant impact on the computed 
pair potential.

\begin{remark}
\label{Rem:compressibility}
\rm
Since $p$-HNCGN fits the data points of the measured radial distribution 
function, it does provide a good fit of the compressibility $\kappa_T$ of the
fluid as well, because the compressibility is given by the Kirkwood-Buff 
integral
\bdm
   \frac{\rho_0}{\beta}\,\kappa_T
   \,=\, 1 \,+\, 4\pi \rho_0\int_0^\infty h(r)\,r^2\dr\,,
\edm
which only depends on the pair correlation function $h=g-1$.
Therefore $p$-HNCGN is able to fit both the compressibility and the
pressure of a fluid to a reasonable accuracy. 
In the pertinent literature this has been considered impossible 
when using isotropic pair potentials,
compare, e.g., Wang, Junghans, and Kremer~\cite{WJK09}.
\fin
\end{remark}

\section{Conclusion}
\label{Sec:Conclusion}
We have determined new generalized Newton schemes for the inverse Henderson
problem, where we approximate the inverse of the Jacobian by the functional
derivative of the hypernetted-chain approximation of the pair potential.
These methods have about the same computational cost per iteration as IBI,
but need much less iterations near phase transitions. In terms of iteration
counts they are competitive to IMC, but the individual iterations are much 
cheaper
than the IMC ones, because no cross-correlations need to be evaluated in the
numerical simulation of the corresponding ensemble of particles.
While these methods turn out to be similar (but not identical)
to the LWR scheme of Levesque, Weis and Reatto, they are more flexible 
by construction, and can easily be modified, e.g., 
to also match the true pressure of the target ensemble.

We finally mention that one can also use the Percus-Yevick 
approximation instead of the hypernetted-chain approximation 
for the derivation of a corresponding generalized Newton method. 
The resulting scheme is very similar to \req{IHNC},
the only difference being that $\varphi_k$ is replaced by
$\varphi_k/y_k$, where $y_k$ is the cavity distribution function associated
with the $k$-th pair potential $u_k$. In our numerical experiments we found
the iteration~\req{IHNC} to perform better near phase transitions of the
truncated and shifted Lennard-Jones potentials than the corresponding
Percus-Yevick recursion, and therefore we have restricted our attention
to the IHNC scheme in this work.

In future work we plan to extend our methods to binary mixtures of different
fluids.

\section*{Appendix: The Wiener lemma}
\label{App:Wiener}
\renewcommand{\thesection}{\Alph{section}}
\setcounter{section}{1}
\setcounter{footnote}{1}
\setcounter{equation}{0}
For $\varrho$ defined in \req{walpha} we have shown in \cite{Hank18c}
that the space $\Lrhorrr$ of all functions $f:\R^3\to\R$, for which
\bdm
   \norm{f}_\Lrhorrr \,=\, 
   \esssup_{R\in\R^3}\,\varrho(|R|)\bigl|f(R)\bigr|\,<\,\infty
\edm
constitutes a Banach algebra with respect to convolution\footnote{Throughout 
this appendix we only consider functions of three variables, whether they 
be radial functions, or not. If $f\in\Lrhorrr$ is radially symmetric, then its
representation defined in $\R^+$ belongs to the Banach space $\Lrho$ 
introduced in \req{Linftyrho}. The three-dimensional Fourier transform 
of $f\in\Lrhorrr$ is denoted by $\fhat$.}. 
We can extend $\Lrhorrr$ to a Banach algebra $\Wrho$ with unit element $e$ 
(given by the delta distribution at the origin), using the canonical norm
\bdm
   \norm{\lambda e+f}_\Wrho \,=\, |\lambda| \,+\, \gamma\norm{f}_{\Lrhorrr}\,,
   \qquad \lambda\in\R, \ f\in \Lrhorrr\,,
\edm
where $\gamma>0$ is a small enough constant to make the norm of $\Wrho$ 
submultiplicative.

The standard Wiener lemma for the Fourier transform starts with a similar 
construction for the Banach algebra $L^1(\R^3)$ and states that if 
$f\in L^1(\R^3)$ is such that $1+\fhat\neq 0$, then 
\be{Wiener}
   (1+\fhat)^{-1} \,=\, 1-\chat
\ee
for some $c\in L^1(\R^3)$ with Fourier transform $\chat$, 
cf., e.g., J\"orgens~\cite{Joer82}\footnote{Note that we have deliberately
denoted this function by $c$; in fact, if $h$ is the radially symmetric
extension of the pair correlation function
and if $f=\rho_0 h$, then $c/\rho_0$ coincides with the 
direct correlation function in the Ornstein-Zernike relation~\req{c}; 
compare \req{c0hatchat}.}.
The weighted Wiener lemma which is required for the proof of 
Propositions~\ref{Prop:hnc} reads as follows.

\begin{lemma}
\label{Lem:Wiener}
Let $f\in \Lrhorrr$ be such that $1+\fhat\neq 0$. Then the function $c$
of \req{Wiener} belongs to $\Lrhorrr$.
If $f$ is a radial function, so is $c$.
\end{lemma}

\begin{proof}
We choose $u$ (not to mix up with the pair potential in the remainder of this
paper) from the standard Schwartz space $\Schwartz$, 
sufficiently close to $f$ in $L^1(\R^3)$ so that $1+\uhat\neq 0$. 
Then we can apply the classical Wiener lemma to deduce that
there exist $c,d\in L^1(\R^3)$ which satisfy \req{Wiener} and
\be{q-App}
   (1+\uhat)^{-1} \,=\, 1-\dhat\,,
\ee
respectively. Moreover, $d\to c$ in $L^1(\R^3)$ as $u\to f$ in $L^1(\R^3)$; 
see \cite{Joer82}. Evidently,
\bdm
   \dhat \,=\, \frac{\uhat}{1+\uhat}\ \in \Schwartz\,,
\edm
and hence, $d\in\Schwartz$, and
\be{w-App}
   w \,=\, (e-d)*(u-f)\ \in \Lrhorrr
\ee
with
\bdm
   \norm{w}_{L^1(\R^3)} 
   \,\leq\, \bigl(1+\norm{d}_{L^1(\R^3)}\bigr)\norm{u-f}_{L^1(\R^3)} \,<\, 1\,,
\edm
provided $u$ is sufficiently close to $f$. In \req{w-App} and below the symbol
$*$ refers to the standard three-dimensional convolution, i.e.,
\bdm
   w(R) \,=\, u(R)-f(R) - \int_{\R^3} d(R-R') \bigl(u(R')-f(R')\bigr)\dR'\,,
   \qquad R\in\R^3\,.
\edm
Since $\norm{w}_{L^1(\R^3)}$ has been shown to be less than 1, 
Corollary~4.3 in \cite{Hank18c} allows to conclude that the series 
\be{Ws}
   \Ws \,:=\, \sum_{n=1}^\infty W_n
\ee
of the $n$-fold autoconvolutions $W_n$ of $w$ converges in $\Lrhorrr$, and hence,
\bdm
   c_0 \,:=\, d - \Ws + d*\Ws \ \in \ \Lrhorrr\,.
\edm
It turns out that this very function $c_0$ coincides with $c$, for we have
\bdm
   \chat_0 \,=\, \dhat \,-\, (1-\dhat)\,\frac{\what}{1-\what}
   \,=\, \frac{\dhat-\what}{1-\what}\,,
\edm
and when inserting~\req{w-App} and \req{q-App} it follows that
\be{c0hatchat}
   \chat_0 \,=\, \frac{\dhat-(1-\dhat)(\uhat-\fhat)}{1-(1-\dhat)(\uhat-\fhat)}
   \,=\, \frac{\uhat-(\uhat-\fhat)}{1+\uhat-(\uhat-\fhat)} 
   \,=\, \frac{\fhat}{1+\fhat} \,=\, \chat\,,
\ee
as has been claimed. This shows that $c\in\Lrhorrr$.

If $f$ is a radial function, so is $\fhat$ and also $\chat$
according to \req{c0hatchat}. Hence, $c$ is a radial function, too.
\end{proof}

Motivated by \req{c0hatchat} we simply write
\be{cWiener}
   c \,=\, f*(e+f)^{-1}
\ee
for the solution $c$ of \req{Wiener} in the sequel. 
For the ease of completeness we also include the following result on
continuous dependence of $c\in\Lrhorrr$.

\begin{lemma}
\label{Lem:Wiener2}
Let $f\in\Lrhorrr$ satisfy the assumptions of Lemma~\ref{Lem:Wiener}, and let 
$c$ be given by \req{cWiener}. If $f_k\in\Lrhorrr$ is sufficiently close 
to $f$ in $\Lrhorrr$, then the Ornstein-Zernike relation~\req{Wiener}
with $f$ replaced by $f_k$ has a well-defined solution $c_k\in\Lrhorrr$,
and there holds
\bdm
   \norm{c_k-c}_\Lrhorrr \,\to\, 0 \qquad \text{as}\qquad
   \norm{f_k-f}_\Lrhorrr \,\to\, 0\,. 
\edm
\end{lemma}

\begin{proof}
We write 
\bdm
   (e+f_k)^{-1} \,=\, (e+f)^{-1}*(e+w_k)^{-1}
\edm
with
\be{A.w}
   w_k \,=\, (e+f)^{-1}*(f_k-f)\,,
\ee
and note that $\norm{w_k}_{L^1(\R^3)}\leq q<1$ for $\norm{f_k-f}_\Lrhorrr$ 
sufficiently small. Using~\req{Wiener} it follows that
\begin{align}
\nonumber
   c_k &\,=\, f_k*(e+f_k)^{-1}
        \,=\, f_k*(e+f)^{-1}*(e+w_k)^{-1}\\[1ex]
\nonumber
       &\,=\, f_k*(e-c)*(e+w_k)^{-1} 
        \,=\, f_k*(e-c)*(e+\Wsk)\\[1ex]
\label{eq:cnbound}
       &\,=\, f_k - f_k*c + (f_k-f_k*c)*\Wsk\,,
\end{align}
where $\Wsk$ is the series~\req{Ws} of the $n$-fold autoconvolutions of $w_k$.
Note that
\be{Wsbound}
   \norm{\Wsk}_\Lrhorrr \,\leq\, C\norm{w_k}_\Lrhorrr
\ee
for some $C>0$ which only depends on the upper bound $q$ of 
$\norm{w_k}_{L^1(\R^3)}$, cf.~\cite{Hank18c}.
Rewriting $c$ as $f*(e-c)$ by virtue of \req{Wiener} and \req{cWiener},
we conclude from \req{cnbound} that
\bdm
   c_k-c \,=\, f_k-f \,-\, (f_k-f)*c \,+\, (f_k-f_k*c)*\Wsk\,,
\edm
and hence, the assertion follows from \req{Wsbound} and \req{A.w}.
\end{proof}

\section*{Acknowledgements}
We are grateful to Gergely T{\'o}th for pointing out to us references
\cite{LWR85,Scho73,Toth07}.


\end{document}